\documentclass[11pt]{article}
\usepackage{geometry}
\usepackage{color}
\usepackage{float}
\usepackage{textcomp}
\usepackage{amsthm}
\usepackage{amsmath}
\usepackage{amssymb}

\usepackage{changes}

\usepackage{mathtools}

\newtheorem{theorem}{Theorem}[section]
\newtheorem{lemma}[theorem]{Lemma}

\newtheorem{proposition}[theorem]{Proposition}

\newcommand{\R}{\mathbb{R}}
\newcommand{\BR}{ {\overline{\mathbb{R}} }  }
\newcommand{\Dom}{\mathrm{Dom}}

\newcommand{\inner}[2]{\langle{#1},{#2}\rangle}

\newcommand{\norm}[1]{\|#1\|}

\newcommand{\tos}{\rightrightarrows} 

\newcommand{\bi}{\begin{itemize}}
\newcommand{\ei}{\end{itemize}}
\newcommand{\ba}{\begin{array}}
\newcommand{\ea}{\end{array}}

\newcommand{\mgap}{\vspace{.1in}}
\newcommand{\vgap}{\vspace{.1in}}

\begin{document}

\title{Regularized HPE-type methods for solving monotone inclusions\\  with improved pointwise  iteration-complexity bounds}
% \title{Improving the pointwise iteration-complexity of the HPE method}
\author{
M. Marques Alves
\thanks{
Departamento de Matem\'atica,
Universidade Federal de Santa Catarina,
Florian\'opolis, Brazil, 88040-900 
({\tt maicon.alves@ufsc.br}).
This work was
done while this author was a postdoc at the 
School of Industrial and Systems Engineering, 
Georgia Institute of Technology, Atlanta, GA, 30332-0205.
The work of this author was partially 
supported by CNPq grants no.
406250/2013-8, 237068/2013-3 and
306317/2014-1
}
\and
R.D.C. Monteiro
    \thanks{School of Industrial and Systems
    Engineering, Georgia Institute of
    Technology, Atlanta, GA, 30332-0205.
    (email: {\tt monteiro@isye.gatech.edu}). The work of this author
    was partially supported by NSF Grants CCF-0808863 and CMMI-0900094
    and ONR Grant ONR N00014-11-1-0062.}
  \and
    B. F. Svaiter\thanks{ IMPA, Estrada Dona Castorina 110, 22460-320 Rio de
    Janeiro, Brazil ({\tt benar@impa.br}). 
The work of this author was partially supported by CNPq grants no.
302962/2011-5, % Prod & Pesq
474996/2013-1, % Universal
and FAPERJ grants
E-26/102.940/2011 and
201.584/2014 (Cientista de Nosso Estado).}
}

\maketitle

\begin{abstract}
 This paper studies 
 the iteration-complexity of new  
 regularized hybrid proximal extragradient 
 (HPE)-type  methods for solving monotone 
 inclusion problems (MIPs). 
 The new (regularized HPE-type) methods essentially consist of instances of the standard  HPE method
 applied to regularizations of the original MIP.
 It is shown that its pointwise iteration-complexity
 considerably improves the one of the HPE method
 while approaches (up to a logarithmic factor) the
 ergodic iteration-complexity of the latter method.
 \end{abstract}

\pagestyle{plain}

%%%%%%%%%%%%%%%%%%%%%%%%%%%%%%%%%%%%%%%%%%%%%%%%%%%%%%%%%%%%%%%%%%%%%%%%%%%%%%%%%%%%%%%%

%\newpage
\section{Introduction}
\label{sec:int}
We consider the \emph{monotone inclusion problem} (MIP)
of finding $x$ such that 
\begin{align}
 \label{eq:mip}
  0\in B(x)
\end{align}
where $B$ is a point-to-set maximal monotone operator. 
One of the most important schemes for solving
MIPs is the proximal point method (PPM),
proposed by Martinet \cite{MR0298899} and further developed
by Rockafellar \cite{roc-mon.sjco76}.
It is an iterative scheme which, in its exact version,
generates a sequence $\{x_k\}$ according to
$x_{k}=(I+\lambda_k B)^{-1}x_{k-1}$ 
(where $\lambda_k>0$ is a regularization
parameter), or equivalently, $x_k$ as the
unique solution of the MIP: 
$0\in \lambda_k B(x)+x-x_{k-1}$. 
Among other results, Rockafellar \cite{roc-mon.sjco76}
proposed inexact versions of the PPM based on a
summable absolute error criterion and subsequently 
Solodov and Svaiter 
\cite{sol.sva-hpe.sva99, sol.sva-hyb.jca99} proposed new inexact variants
based on a hybrid proximal extragradient (HPE) relative error criterion. 
In each step, the variants proposed and studied in 
\cite{sol.sva-hpe.sva99}, namely the
HPE method,
computes $\lambda=\lambda_k>0$ and a triple
$(y,b,\varepsilon)=(y_k,b_k,\varepsilon_k)$  
satisfying
\begin{align}
\label{eq:ec3}
 b\in B^{[\varepsilon]}(y),\qquad
  \norm{\lambda b+y-x}^2+
  2\lambda\varepsilon
    \leq \sigma^2\norm{y-x}^2,
\end{align}
where $x=x_{k-1}$ is the current iterate,
$\sigma\in [0,1)$ is a relative error tolerance
and $B^{[\varepsilon]}$ denotes the
$\varepsilon$-enlargement 
\cite{bur.ius.sva-enl.svva97} of $B$. Moreover, 
instead of choosing $y$ as the next iterate, the HPE method
computes $x_+=x_k$ by means of an
extragradient step $x_+=x-\lambda b$.

The iteration-complexity of the HPE method
was established in \cite{mon.sva-hpe.siam10} with regards to the following
termination criterion in terms of precisions $\bar\rho>0$ and $\bar\varepsilon>0$:
find a triple $(y,b,\varepsilon)$ such that
\begin{equation}
\label{eq:appsol.c}
 b\in B^{[\varepsilon]}(y),\quad 
\norm{b}\leq \bar\rho,\quad  \varepsilon\leq \bar\varepsilon.
\end{equation}
Assuming that the sequence of stepsizes
$\{\lambda_k\}$ in the HPE method 
is bounded below by some constant 
$\underline{\lambda}>0$, the pointwise   
iteration-complexity result of 
\cite{mon.sva-hpe.siam10} guarantees that
the most recent triple $(y,b,\varepsilon)$
satisfying \eqref{eq:ec3} will eventually 
satisfy the termination criterion given in 
\eqref{eq:appsol.c}
in at most 
$\mathcal{O}\left(\max\left\{
d_0^2/\underline{\lambda}^2\bar\rho^2,
d_0^2/\underline{\lambda}\bar\varepsilon\right\}
\right)$ iterations where $d_0$ denotes the distance
of the initial iterate $x_0$ to the solution set 
of \eqref{eq:mip}. 
Moreover, under the same condition on the sequence
of stepsizes $\{\lambda_k\}$, an ergodic 
iteration-complexity result of 
\cite{mon.sva-hpe.siam10} shows that an ergodic triple constructed from all previous generated triples
satisfying \eqref{eq:ec3} will eventually 
satisfy \eqref{eq:appsol.c} in at most 
$\mathcal{O}\left(\max\left\{
d_0/\underline{\lambda}\bar\rho,
d_0^2/\underline{\lambda}\bar\varepsilon\right\}
\right)$ iterations.
Clearly, the ergodic iteration-complexity is better 
than the pointwise one by
a factor of $\mathcal{O}\left
(\max\{1,d_0/\underline{\lambda} \bar\rho\}\right)$.

Our main goal in this paper is to present 
regularized HPE-type methods for 
solving \eqref{eq:mip}
which essentially consists of instances of the 
HPE method applied to the regularized
MIP
\begin{align}
 \label{eq:mip.mu}
0 \in B(x) + \mu (x-x_0)
\end{align}
where $\mu>0$ and $x_0$ is an initial point.
In particular, it is shown that
a certain version of the regularized HPE method which
dynamically adjusts $\mu>0$ solves 
\eqref{eq:mip} in at most
\begin{align}
 \mathcal{O}\left(\left(
\dfrac{d_0}
{\underline{\lambda}\bar \rho}
+1\right)
\left[1+
\max\left\{\log^+\left(
\dfrac{d_0}
{\underline{\lambda}\bar \rho}\right),
\log^+\left( \frac{d_0} {\underline{\lambda}\bar \varepsilon}\right)
\right\}
\right]\right)
\end{align}
iterations. This pointwise iteration-complexity 
bound  considerably improves the one for the usual HPE method.
Also, note that it differs 
from the ergodic one for the
usual HPE method 
by only a  logarithmic factor.
 Finally, we discuss specific instances of the regularized HPE method
which are based on  Tseng's modified forward-backward splitting (MFBS) 
method \cite{MR1741147}
and Korpelevich's extragradient method 
\cite{MR0451121}.

\vspace{0.2cm}
\noindent
{\bf Previous most related works.} 
In the context of variational inequalities (VIs), 
Nemirovski \cite{Nem05-1} has established the 
ergodic iteration-complexity of an extension of 
Korpelevich's method, namely, 
the mirror-prox algorithm, under
the assumption that the feasible set 
of the problem is bounded.
Nesterov \cite{Nesterov-ext} 
proposed a new dual 
extrapolation algorithm for solving VIs whose
termination depends on the 
guess of a ball centered at the initial iterate.
Applications 
of the HPE method to the 
iteration-complexity analysis of 
several zeroth-order (or, in the context of
optimization, first-order) 
methods for solving monotone VIs, MIPs 
and saddle-point problems were
discussed by Monteiro and Svaiter 
in \cite{mon.sva-hpe.siam10} 
and in the subsequent papers 
\cite{mon.sva-com.siam10,mon.sva-ite.siam13}.
The HPE method was also used to study the 
iteration-complexities of first-order (or, in the
context of optimization, second-order) 
methods for solving either a 
monotone nonlinear equation (see
Section 7 of \cite{mon.sva-hpe.siam10}) 
and, more generally, a monotone VI 
(see \cite{mon.sva-ite.siam12}).

\vspace{0.2cm}
\noindent
{\bf Organization of the paper.} 
Section \ref{sec:pre} contains two subsections.
Subsection \ref{sec:bas} presents the notation
as well as some basic concepts about convexity
and maximal monotone operators.
Subsection \ref{sec:smhpe} is devoted to the study of
a specialization of the HPE method
for solving inclusions whose underlying
operator is written as a sum of a (maximal)
monotone and a strongly monotone operator.
Section \ref{sec:imp.hpe}  presents 
the main contributions
of the paper, namely, the presentation of two new regularized HPE
methods (a static one and a dynamic one) as well as its complexity analysis.
Section \ref{sec:app} discusses two specific instances of the dynamic regularized HPE
method of Section \ref{sec:imp.hpe} based on
Tseng's MFBS method and Korpelevich's extragradient
method. Finally, the appendix
presents the proofs of some results in Subsection 
\ref{sec:smhpe}.

%\vspace{0.2cm}
%\noindent
%{\bf General notation.} 
%Throughout this paper, we assume that $X$ 
%is a (real) finite dimensional vector
%space with inner product $\inner{\cdot}{\cdot}$ 
%and induced norm $\|\cdot\|:=\sqrt{\inner{\cdot}{\cdot}}$.
%For $t>0$, we let $\log^+(t):=\max\{\log(t),0\}$.

\section{Preliminaries}
\label{sec:pre}

This section discusses some preliminary results
which will be used throughout the paper.
Subsection \ref{sec:bas} presents the
general notation and some basic concepts about
convexity, maximal monotone operators, and
related issues. 
Subsection \ref{sec:smhpe} describes a 
special version of the HPE method 
introduced in \cite{sol.sva-hpe.sva99}
for solving monotone inclusions whose 
underlying operators consist of the sum of a (maximal) 
monotone and a strongly
(maximal) monotone operator.

\subsection{Basic concepts and notation}
\label{sec:bas}
For $t>0$, we let $\log^+(t):=\max\{\log(t),0\}$.
Let also $X$ be a finite-dimensional 
real vector space with
inner product $\inner{\cdot}{\cdot}$ and induced norm 
$\|\cdot\|:=\sqrt{\inner{\cdot}{\cdot}}$.
Given a set-valued operator $A:X\tos X$, 
its \emph{graph} and
\emph{domain} are, respectively, 
$\mbox{Gr}(A):=\{(x,v)\in X\times X\,:\, v\in A(x)\}$
and $\mbox{Dom}(A):=\{x\in X\,:\, A(x)\neq \emptyset\}$.
The \emph{inverse} of $A:X\tos X$ is 
$A^{-1}:X\tos X$, 
$A^{-1}(v):=\{x\;:\; v\in A(x)\}$.
The \emph{sum} of two set-valued 
operators $A, B:X\tos X$ is defined by
$
 A+B:X\tos X,\;\;(A+B)(x):=
\{a+b\in X\,:\,a\in A(x),\; b\in B(x)\}.
$

An operator $A:X\tos X$ is \emph{$\mu$-strongly
monotone} if $\mu\geq 0$ and
\begin{align}
\label{eq:sm.c}
 \inner{v-v'}{x-x'}\geq \mu\norm{x-x'}^2\quad \forall (x,v),(x',v')\in \mbox{Gr}(A).
\end{align}
If $\mu=0$ in the above inequality, then $A$ is said to be a
\emph{monotone} operator.
Moreover, $A:X\tos X$ is \emph{maximal monotone} if it is monotone and maximal in the following sense:
if $B:X\tos X$ is monotone and $\mbox{Gr}(A)\subset \mbox{Gr}(B)$, then $A=B$.
The \emph{resolvent} of a maximal monotone operator $A:X\tos X$
with parameter $\lambda>0$ is $(I+\lambda A)^{-1}$. It follows directly
from this definition that $y=(I+\lambda A)^{-1} x$ 
if and only if $(x-y)/\lambda\in A(y)$.
It is easy to see that if $A:X\tos X$ is $\mu$-strongly monotone and $B:X\tos X$ is monotone,
then the sum $A+B$ is also $\mu$-strongly monotone. In particular, the sum of two monotone
operators is also a monotone operator.

The \emph{$\varepsilon$-enlargement} 
\cite{bur.ius.sva-enl.svva97} of a 
maximal monotone operator $B:X\tos X$ is defined
by $B^{[\varepsilon]}:X\tos X$,
\begin{align}
\label{eq:def.eps}
 B^{[\varepsilon]}(x)
 :=\{v\in X\,:\,\inner{v-v'}{x-x'}\geq -\varepsilon,\;\;\forall (x',v')\in \mbox{Gr}(B)\}.
\end{align}

The following summarizes some useful properties of $B^{[\varepsilon]}$.

\begin{proposition}
\label{pr:teps.pr}
Let $A, B:X\tos X$ be maximal monotone operators.  Then,
\begin{itemize}
\item[\emph{(a)}] if $\varepsilon_1\leq \varepsilon_2$, then
$A^{[\varepsilon_1]}(x)\subseteq A^{[\varepsilon_2]}(x)$ for every $x \in X$;
\item[\emph{(b)}] $A^{[\varepsilon']}(x)+(B)^{[\varepsilon]}(x) \subseteq
(A+B)^{[\varepsilon'+\varepsilon]}(x)$ for every $x \in X$ and
$\varepsilon, \varepsilon'\geq 0$;
\item[\emph{(c)}] $A$ is monotone if, and only if, $A  \subseteq A^{[0]}$;
\item[\emph{(d)}] $A$ is maximal monotone if, and only if, $A = A^{[0]}$;
%\item [(e)] if $f:X\to\overline{\R}$ is convex, proper and lower
%  semicontinuous, then
%  $\partial_\varepsilon f(x)\subset (\partial f)^{[\varepsilon]}(x)$ for
%  any $\varepsilon \geq 0$ and $x\in X$.
\end{itemize}
\end{proposition}

Recall that the 
\emph{$\varepsilon$-subdifferential} of a 
proper closed convex function $f:X\to \BR$
is defined at $x\in X$ by
$\partial_{\varepsilon}f(x):=\{v\in X\,:\,f(x')\geq f(x)+\inner{v}{x'-x}-\varepsilon\;\;\forall x'\in X\}$.
When $\varepsilon=0$, then $\partial f_0(x)$ 
is denoted by $\partial f(x)$
and is called the \emph{subdifferential} of $f$ at $x$.
The simplest example of subdifferential is 
given by considering indicator functions
of closed convex sets. Given a closed convex set 
$\mathcal{X}\subset X$ its \emph{indicator function}
is denoted by $\delta_{\mathcal{X}}$ and is defined by
$\delta_{\mathcal{X}}(x):=0$ if $x\in \mathcal{X}$
and $\delta_{\mathcal{X}}(x):=\infty$ otherwise.
The \emph{normal cone} 
of $\mathcal{X}$
is defined by $N_\mathcal{X}:=
\partial \delta_{\mathcal{X}}$. 
We also define the \emph{projection} on $\mathcal{X}$ by 
$P_\mathcal{X}:=(I+N_{\mathcal{X}})^{-1}$.

\subsection
{Solving inclusions with strongly monotone operators}
\label{sec:smhpe}

In this subsection, we 
consider the MIP
\begin{align}
\label{eq:inc.p}
 0\in A(x)+B(x)
\end{align}
where the following assumptions hold:
\begin{itemize}
\item[A.1)] $A:X\tos X$ is a 
$\mu$-strongly maximal monotone operator 
for some $\mu>0$ (see \eqref{eq:sm.c}); 
\item[A.2)] $B:X\tos X$ is maximal monotone;
\item[A.3)] the solution set of~\eqref{eq:inc.p}, i.e., $(A+B)^{-1}(0)$, is nonempty.
\end{itemize}

We next state a specialized  HPE method  for solving \eqref{eq:inc.p} under the  assumptions stated above.
It will be used later on in Section \ref{sec:imp.hpe} 
to describe regularized HPE methods 
for general MIPs whose
pointwise iteration-complexities improve 
the ones for the usual HPE method 
(see \cite{mon.sva-hpe.siam10}).

%a specialization of the HPE
%method (see \cite{mon.sva-hpe.siam10})
%for solving \eqref{eq:inc.p}.
%

\mgap
\mgap

\noindent
\fbox{
\begin{minipage}[h]{6.6 in}
{\bf Algorithm 1:} {A specialized HPE method for solving strongly MIPs}
\begin{itemize}
\item[(0)] Let $x_0\in X$ and $\sigma\in [0,1)$ be 
given
and set $k=1$;
\item[(1)] choose $\lambda_k>0$ and find $y_k,v_k\in X$, $\sigma_k\in[0,\sigma]$, and $\varepsilon_k\geq 0$ 
       such that
       \begin{equation}\label{eq:ec}
        v_k\in A(y_k)+B^{[\varepsilon_k]}(y_k)\,,\;\;\|\lambda_kv_k+y_k-x_{k-1}\|^2+2\lambda_k\varepsilon_k\leq 
        \sigma_k^2\|y_k-x_{k-1}\|^2;
       \end{equation}
\item[(2)] set 
          \begin{equation}
 \label{eq:es}
 x_k=x_{k-1}-\lambda_kv_k,
          \end{equation}
          let $k\leftarrow k+1$ and go to step 1.
\end{itemize}
\noindent
{\bf end}
\end{minipage}
}

\vgap
\vgap

We now make some remarks about Algorithm 1.
First, it can be easily checked that if $\sigma=0$
then Algorithm 1 reduces to the exact proximal point method
(PPM) for solving \eqref{eq:inc.p}, i.e.,
\begin{align*}
  x_k=(\lambda_k(A+B)+I)^{-1}x_{k-1}\qquad 
\forall k\geq 1.
\end{align*}
Second, since $A(y)+B^{[\varepsilon]}(y)\subset  (A+B)^{[\varepsilon]}(y)$ for every $y$ in view of Proposition \ref{pr:teps.pr}(b),
it follows that Algorithm 1 is a special instance
of the HPE method studied in 
\cite{mon.sva-hpe.siam10}. 
Third, like in the HPE method, step 1 of Algorithm
1 does not specify how to compute the stepsize
$\lambda_k$ and the triple 
$(y_k,v_k,\varepsilon_k)$. Their computation 
will depend on the
instance of the method under consideration.
%(see Section \ref{sec:app} for  an instance based on a variant of the Tseng's modified forward-backward method).

The next result derives convergence rates
for the sequences $\{v_k\}$ and $\{\varepsilon_k\}$ generated by Algorithm~1 under the assumption
that the sequence of stepsizes $\{\lambda_k\}$
is bounded away from zero.
Its proof is given  in Appendix \ref{sec:app01}.

\begin{proposition} \label{pr:mmm}
Let %\textcolor{red}{$d_k$} and
$d_0$ denote the distance
of %$x_k$ and
$x_0$ %, respectively, 
to the solution set of~\eqref{eq:inc.p} and
define
\begin{align}
\label{eq:def.theta}
 \theta :=  \left( \frac{1}{2\underline{\lambda}\mu} +
 \frac{1}{1-\sigma^2} \right)^{-1} \in (0,1).
\end{align}
Assume that  $\lambda_k\geq \underline{\lambda}>0$ 
for every $k\geq 1$.
Then, for every $k\geq 1$, $v_k\in A(y_k)+B^{[\varepsilon_k]}(y_k)$,
\[
 \|v_k\| \le \sqrt{\dfrac{1+\sigma}{1-\sigma}}
 \left(\frac{(1-\theta)^{(k-1)/2}}{\underline{\lambda}} \right) d_0,
   \qquad
 \varepsilon_k \le \frac{\sigma^2}{2(1-\sigma^2)}
 \left(\frac{(1-\theta)^{k-1}}{\underline{\lambda}} \right) d_0^2.
\]
\[
 \|x^*-x_k\| \le (1-\theta)^{k/2} \|x^*-x_0\|
 %\qquad \textcolor{red}{d_k \le (1-\theta)^{k/2} d_0}
\quad \forall x^* \in (A+B)^{-1}(0).
\]
\end{proposition}

%\newpage
\section{
{Regularized  HPE 
methods for solving MIPs}}
\label{sec:imp.hpe}

This section presents 
regularized HPE-type methods for 
solving MIPs whose pointwise iteration-complexity 
is superior to
the one for the original HPE method 
(see \cite{mon.sva-hpe.siam10}). It is shown that the new pointwise 
bound is worse than the ergodic one for the
original HPE method by only a logarithmic factor.

This section considers the 
MIP \eqref{eq:mip}
where $B:X\tos X$ is a point-to-set maximal monotone operator such that $B^{-1}(0)\neq \emptyset$,
and discusses regularized HPE-type methods which, for a 
given point $x_0 \in X$, consist of  solving MIPs parametrized by a scalar $\mu>0$ as in \eqref{eq:mip.mu}.
%
%\textcolor{blue}{
%\begin{align}
%\label{eq:mip.mu}
% 0\in B(x)+\mu(x-x_0).
%\end{align}
%
%}
Observe that \eqref{eq:mip.mu} is a regularized version of \eqref{eq:mip}. Its operator
is $\mu$-strongly monotone and approaches the one of \eqref{eq:mip} as $\mu>0$ approaches zero.
Clearly, \eqref{eq:mip.mu}  is a special case of~\eqref{eq:inc.p}
with $A(x)=\mu(x-x_0)$ and its solution set
is a singleton by Minty's theorem 
\cite{min-mon.duke62}.

We denote the distance of $x_0$ to the solution sets of \eqref{eq:mip} and \eqref{eq:mip.mu} by $d_0$ and $d_\mu$,
respectively. Clearly, 
\begin{align}
  \label{eq:dmuxmu}
  d_\mu=\norm{x^*_\mu-x_0}
\end{align}
where $x^*_\mu$ denotes the unique solution of \eqref{eq:mip.mu}, i.e., 
$x^*_\mu=(\mu^{-1}B+I)^{-1}(x_0)$.

The following simple technical result relates $d_\mu$ with $d_0$.

\begin{lemma}
\label{lm:dz}
For every $\mu>0$, $d_\mu\leq d_0$.
\end{lemma}
\begin{proof}
 Let $x^*$ be the projection of $x_0$ onto $B^{-1}(0)$.
 Since $0\in B(x^*)$ and 
 $\mu(x_0-x^*_\mu)\in B(x^*_\mu)$, 
 the monotonicity of $B$ and the fact that
 $\mu>0$ imply that
%
%\begin{align*}
 $\inner{x^*-x^*_\mu}{x^*_\mu -x_0}\geq 0$.
%\end{align*}
%
Therefore, 
\begin{align*}
  d_0^2=\norm{x^*-x_0}^2=
 \norm{x^*-x^*_\mu}^2+
 2\inner{x^*-x^*_\mu}{x^*_\mu-x_0}+
  \norm{x^*_\mu-x_0}^2\geq
   \norm{x^*-x^*_\mu}^2+d_\mu^2
\end{align*}
and the conclusion follows.
\end{proof}

We now state a $\mu$-regularized 
HPE method for solving \eqref{eq:mip} 
which is simply Algorithm 1 (with $A(\cdot)=\mu (\cdot-x_0)$) applied to MIP \eqref{eq:mip.mu}
but with a termination criterion added.

\vgap
\vgap

\noindent
\fbox{
\begin{minipage}[h]{6.6 in}
{\bf Algorithm 2:}
{A static $\mu$-regularized HPE method for solving \eqref{eq:mip}.} 
\begin{itemize}
\item[] \mbox{Input:} $(x_0,\sigma,\mu,\rho,
\varepsilon)\in X\times [0,1)\times \R_{++}
\times \R_{++}\times \R_{++}$; 
\item[(0)] set $k=1$;
\item[(1)] choose $\lambda_k>0$ and find 
$(y_k,b_k,\varepsilon_k)\in X\times X\times \R_+$  
       such that
       \begin{equation}\label{eq:ec.2}
b_k\in B^{[\varepsilon_k]}(y_k),\quad
\|\lambda_k\left[b_k+\mu(y_k-x_0)\right]+y_k-x_{k-1}\|^2+2\lambda_k\varepsilon_k\leq 
        \sigma^2\|y_k-x_{k-1}\|^2;
       \end{equation}
\item[(2)] if $\norm{b_k+\mu(y_k-x_0)}> \rho$
or $\varepsilon_k> \varepsilon$,
then set
          \begin{equation}
					\label{eq:es.2}
					x_k=x_{k-1}-\lambda_k
\left[b_k+\mu(y_k-x_0)\right],
          \end{equation}
           
and $k\leftarrow k+1$, and go to step 1;
otherwise, stop the algorithm and output
$(y_k,b_k,\varepsilon_k)$. 
\end{itemize}
\noindent
{\bf end}
\end{minipage}
}

\vgap
\vgap
We now make some remarks about Algorithm 2.
First, it is the special case of Algorithm 1 in which $ A(\cdot)=\mu(\cdot-x_0)$,
and hence solves the MIP \eqref{eq:mip.mu}.
Second, since Subsection~\ref{sec:smhpe} only deals with
convergence rate bounds, a stopping criterion was not added to Algorithm 1.
In contrast, Algorithm~2 incorporates
a stopping criterion (see step 2 above) based on which  its iteration-complexity bound is derived
in Proposition~\ref{th:c.alg2} and  Theorem~\ref{cr:c.alg2} below.
Third, it is shown in Theorem \ref{cr:c.alg2}(b) 
that Algorithm~2 solves MIP \eqref{eq:mip} if $\mu$ is chosen sufficiently small.

%\newpage
\begin{proposition}
\label{th:c.alg2}
Assume that 
$\lambda_k\geq \underline{\lambda}>0$
for all $k\geq 1$
and let $d_\mu$ be as in \eqref{eq:dmuxmu}.
Then, Algorithm 2 
with input 
$(x_0,\sigma,\mu,\rho,\varepsilon)$ terminates 
in at most
\begin{align}
\label{eq:it.c}
 \left(\dfrac{1}{2\underline{\lambda}\mu}
 +\dfrac{1}{1-\sigma^2}\right)
\left[ 2 + 
\max\left\{\log^+\left(
\left[\dfrac{1+\sigma}{1-\sigma}\right]
\dfrac{d_\mu^2}
{\underline{\lambda}^2\rho^2}\right),
\log^+\left(\dfrac{\sigma^2d_\mu^2}{2(1-\sigma^2)
\underline{\lambda}\varepsilon}\right)
\right\} \right]
% \mathcal{O}\left(1+\left(\dfrac{1}{\underline{\lambda}\mu}
% +\dfrac{1}{1-\sigma^2}\right)
%\max\left\{\log^+\left(
%\left[\dfrac{1+\sigma}{1-\sigma}\right]\dfrac{d_\mu}
%{\underline{\lambda}\rho}\right),
%\log^+\left(\dfrac{\sigma^2 d_\mu^2}
%{(1-\sigma^2)\underline{\lambda}\varepsilon}\right)
%\right\}\right)
\end{align}
iterations with a triple  $(y_k,b_k,\varepsilon_k)$ which,
in addition to satisfying the stopping criterion in step 2 of Algorithm 2, namely,
\begin{equation} 
\label{eq:stop}
\norm{b_k+\mu(y_k-x_0)}\leq \rho, \quad 
\varepsilon_k\leq \varepsilon,
\end{equation}
it also satisfies the inequalities
\begin{align}
& \| y_k - x_0\| 
\le \left(1+\dfrac{1}{\sqrt{1-\sigma^2}}\right)d_\mu 
\le \left(1+\dfrac{1}{\sqrt{1-\sigma^2}}\right)d_0 , 
\label{eq:bound.aux} \\
\label{eq:bound.b}
&  \norm{b_k}\leq \rho+
\mu\left(1+\dfrac{1}{\sqrt{1-\sigma^2}}\right)
d_\mu\leq \rho +
\mu\left(1+\dfrac{1}{\sqrt{1-\sigma^2}}\right)d_0.
\end{align}
\end{proposition}
\begin{proof}
%
%We claim that the number of iterations 
%performed by Algorithm 2 is bounded by
%
%\[
% \left(\dfrac{1}{2\underline{\lambda}\mu}
% +\dfrac{1}{1-\sigma^2}\right)
%\left[ 2 + 
%\max\left\{\log^+\left(
%\left[\dfrac{1+\sigma}{1-\sigma}\right]
%\dfrac{d_\mu^2}
%{\underline{\lambda}^2\rho^2}\right),
%\log^+\left(\dfrac{\sigma^2d_\mu^2}{2(1-\sigma^2)
%\underline{\lambda}\varepsilon}\right)
%\right\} \right],
%\]
%
%from which the iteration-complexity order
%\eqref{eq:it.c} trivially follows. Indeed, 
To prove \eqref{eq:it.c} assume
that Algorithm 2 has not terminated at the $k$-th
iteration, and define $v_k=b_k+\mu(y_k-x_0)$. 
Then, either
$\norm{v_k}>\rho$ or $\varepsilon_k>\varepsilon$.
Assume first that $\norm{v_k}>\rho$.
Since Algorithm 2 is a special case of Algorithm 1
applied to MIP \eqref{eq:mip.mu}
with $A(x)=\mu(x-x_0)$
and $v_k$ as above,
the latter assumption 
and Corollary~\ref{pr:mmm} imply that
\[
 \rho<\|v_k\| \le \sqrt{\dfrac{1+\sigma}{1-\sigma}}
 \left(\frac{(1-\theta)^{(k-1)/2}}{\underline{\lambda}}
 \right) d_\mu
\]
where $\theta$ is defined in \eqref{eq:def.theta}.
Rearranging this inequality, taking logarithms of both sides of the resulting  inequality and using
the fact that
$\log(1-\theta)\leq -\theta$,  we conclude that
\[
 k< 1+\theta^{-1}\log
\left(\left[\dfrac{1+\sigma}{1-\sigma}\right]\dfrac{d_\mu^2}
{\underline{\lambda}^2\rho^2}\right).
\]
If, on the other hand, $\varepsilon_k>\varepsilon$,
we conclude by using a similar reasoning that
\[
 k< 1+\theta^{-1}\log
\left(\dfrac{\sigma^2 d_\mu^2}
{2(1-\sigma^2)\underline{\lambda}\varepsilon}\right).
\]
From the above two observations and the fact that $\theta < 1$ in view of \eqref{eq:def.theta},
\eqref{eq:it.c} follows. 

To prove \eqref{eq:bound.aux}, note that
Lemma 2.1(5) of \cite{prep-impa-benar2015}, Corollary \ref{pr:mmm}
and \eqref{eq:dmuxmu} imply that
\[
\norm{y_k-x^*_\mu} \leq
\dfrac{\norm{x_{k-1}-x^*_\mu}}{\sqrt{1-\sigma^2}}
\le \dfrac{(1-\theta)^{(k-1)/2}}
{\sqrt{1-\sigma^2}} d_\mu
\leq \dfrac{1}{\sqrt{1-\sigma^2}}d_\mu,
\]
and hence that
\begin{align*}
%\label{eq:y.ineq}
\norm{y_k-x_0}\leq \norm{y_k-x^*_\mu}+\norm{x^*_\mu-x_0}
\leq 
\left(1+\dfrac{1}{\sqrt{1-\sigma^2}}\right)d_\mu.
%\label{eq:th.100}
\end{align*}
The latter conclusion and Lemma \ref{lm:dz}
yield \eqref{eq:bound.aux}. 
To finish the proof, note that \eqref{eq:bound.b}
follows from the first inequality in \eqref{eq:stop}, 
the triangle inequality and \eqref{eq:bound.aux}.
\end{proof}

The complexity results presented in this paper will
consist in establishing bounds in the number of
iterations to obtain a triple $(y,b,\varepsilon)$ 
satisfying \eqref{eq:appsol.c}, 
%
%\textcolor{blue}{
%\begin{align}
%\label{eq:appsol.c2}
% b\in B^{[\varepsilon]}(y),\quad \norm{b}\leq \bar\rho\;\;\mbox{and}\;\;\varepsilon\leq \bar\varepsilon,
%\end{align}
%
%}
for given precisions $\bar\rho>0$ and 
$\bar \varepsilon>0$.  
%Such
%bounds will depend on the distance of the initial
%guess (for the methods presented here) to the
%solution set of~\eqref{eq:inc.p} and on the 
%parameter $\mu$.

The following result shows that Algorithm 2 solves the MIP \eqref{eq:mip} when $\mu>0$
is chosen sufficiently small.

\begin{theorem}
\label{cr:c.alg2}
Assume that 
$\lambda_k\geq \underline{\lambda}>0$
for all $k\geq 1$
%and a scalar $\mathcal{D}_0\geq d_0$  is given,
and let 
a tolerance pair
 $(\bar\rho,\bar\varepsilon)\in \R_{++}\times \R_{++}$
be given.
Then, the following statements hold:
\begin{itemize}
\item[\emph{(a)}]
for any $\rho\in (0,\bar\rho)$ and 
$\mathcal{D}_0>0$,
Algorithm 2 with input 
$(x_0,\sigma,\mu,\rho,\varepsilon)$
where
\begin{align}
\label{eq:def.mu.e}
 \mu=\mu(\mathcal{D}_0,\rho) := \dfrac{\bar\rho-\rho}
{\left[1+\dfrac{1}{\sqrt{1-\sigma^2}}\right]
\mathcal{D}_0},
\quad
\varepsilon=\bar\varepsilon
\end{align}
terminates 
in at most
\begin{align}
\label{eq:it.cc}
\left(\frac{\left[1+1/\sqrt{1-\sigma^2}\right]
\mathcal{D}_0}{2\underline{\lambda}(\bar\rho-\rho)}
 +\dfrac{1}{1-\sigma^2}\right)
\left[ 2+\max\left\{\log^{+}\left(
\left[\dfrac{1+\sigma}{1-\sigma}\right]
\dfrac{d_0^2}
{\underline{\lambda}^2\rho^2} \right),
\log^+\left(\dfrac{\sigma^2 d_0^2}
{2(1-\sigma^2)\underline{\lambda}\bar 
\varepsilon}\right)
\right\} \right]
\end{align}
iterations;
\item[\emph{(b)}]
 if $\mathcal{D}_0 \ge d_0$, 
then Algorithm 2 with the above input terminates 
with a triple $(y_k,b_k,\varepsilon_k)$
satisfying 
\begin{align}
\label{eq:sev}
b_k\in B^{[\varepsilon_k]}(y_k), \quad 
\norm{b_k}\leq \bar\rho,  \quad \varepsilon_k\leq 
\bar\varepsilon, \quad \mu \|y_k-x_0\| 
\le \bar \rho - \rho.
\end{align}
\end{itemize}
\end{theorem}
\begin{proof}
Note that \eqref{eq:it.cc} 
follows from \eqref{eq:it.c}, \eqref{eq:def.mu.e} and
Lemma \ref{lm:dz}. 
Using the second inequalities in 
\eqref{eq:bound.aux} and 
\eqref{eq:bound.b}
and the first identity in \eqref{eq:def.mu.e} 
we find
\[
 \max\left\{\norm{b_k}-\rho,
 \mu\norm{y_k-x_0}\right\}\leq
\frac{d_0}
{\mathcal{D}_0} (\bar \rho-\rho). 
\]
Thus, if $\mathcal{D}_0\geq d_0$, then
the latter inequality yields 
the second and the fourth inequalities
in \eqref{eq:sev}.
The inclusion and the third inequality in
\eqref{eq:sev} follow from \eqref{eq:ec.2}
and \eqref{eq:stop}, respectively.
\end{proof}

We now make two remarks about Theorem \ref{cr:c.alg2}.
First, if $\sigma \in [0,1)$ is such that $(1-\sigma)^{-1}={\cal O}(1)$, an upper bound ${\cal D}_0 \ge d_0$ such that
${\cal D}_0 = {\cal O}(d_0)$ is known, and $\rho$ is set to $\bar \rho/2$,
then the complexity bound \eqref{eq:it.cc}  is
\begin{align}
\label{eq:it.c3}
 \mathcal{O}\left(\left(
\dfrac{d_0}
{\underline{\lambda}\bar \rho}
+1\right)
\left[1+
\max\left\{\log^+\left(
\dfrac{d_0}
{\underline{\lambda}\bar \rho}\right),
\log^+\left( \frac{d_0} {\underline{\lambda}\bar \varepsilon}\right)
\right\}
\right]\right).
\end{align}
Second, in general an upper bound 
${\cal D}_0$ as in the first remark is not 
known and in such case the bound
\eqref{eq:it.cc} can be much worse than the one above when ${\cal D}_0 >> d_0$.

In the remaining part of this section, we
consider the case where an upper bound ${\cal D}_0 \ge d_0$  such that ${\cal D}_0 = {\cal O}(d_0)$
is not known and describe a scheme
based on Algorithm 2 whose iteration-complexity 
order is equal to  \eqref{eq:it.c3}.

\vgap
\vgap

\noindent
\fbox{
\begin{minipage}[h]{6.6 in}
{\bf DR-HPE:} A dynamic regularized HPE method for solving \eqref{eq:mip}.
\begin{itemize}
\item[(0)] Let $x_0\in X$,
$\sigma\in [0,1)$, $\bar{\lambda}>0$ 
and a tolerance 
pair $(\bar \rho,\bar \varepsilon) \in 
\R_{++} \times \R_{++}$ be given and
choose $\rho\in (0,\bar\rho)$; set
\begin{align}
\label{eq:def.dz}
\mathcal{D}_0 = \overline{\mathcal{D}}_0:=
\dfrac{2\bar{\lambda}(\bar\rho-\rho)}
{(1-\sigma^2)\left(1+1/
\sqrt{1-\sigma^2}\right)};
\end{align}
\item[(1)] set $\mu=\mu(\mathcal{D}_0,\rho)$ where 
$\mu(\cdot,\cdot)$ is defined in 
\eqref{eq:def.mu.e}  
and call Algorithm 2 with input
$(x_0,\sigma,\mu,\rho,\bar \varepsilon)$
to obtain as output $(y,b,\varepsilon)$;
\item[(2)]
if $\mu\norm{y-x_0}\leq \bar\rho-\rho$ 
then stop and output $(y,b,\varepsilon)$;
else, set $\mathcal{D}_0 \leftarrow 2\mathcal{D}_0$
and go to step 1. 
\end{itemize}
\noindent
{\bf end}
\end{minipage}
}
\vgap
\vgap
\vgap

Each iteration of DR-HPE (referred to as an outer iteration) invokes Algorithm 2, and hence performs a certain number of iterations of the latter method (called  inner iterations)
which is bounded by \eqref{eq:it.cc}. The following result gives the overall inner-iteration-complexity of DR-HPE 
in terms of $d_0$, $\bar \lambda$,
$\rho$, $\bar \rho$ and $\bar \varepsilon$.

\begin{theorem}
\label{th:main}
Let $d_0$ denote the distance of $x_0$ to the 
solution set of \eqref{eq:mip} and
assume that the proximal stepsize in every inner iteration of
\emph{DR-HPE} is bounded below by a constant
$\underline{\lambda}>0$.
Then, \emph{DR-HPE} with input 
$(x_0,\sigma,\bar \lambda, 
(\bar \rho,\bar \varepsilon), \rho) \in 
X \times [0,1) \times \R_{++} \times 
\R^2_{++} \times \R_{++}$
such that $\rho \in (0,\bar \rho)$ and $(1-\sigma)^{-1}= {\cal} O(1)$ finds a triple 
$(y,b,\varepsilon)$ satisfying
\[
 b\in B^{[\varepsilon]}(y),\quad \norm{b}\leq \bar\rho,
\quad \varepsilon\leq \bar\varepsilon
\]
in at most
\begin{align}
\label{eq:c.alg3}
 \mathcal{O}\left(\left( 1 + \frac{\bar \lambda}{\underline \lambda} \right) \left(
\dfrac{d_0}
{\bar{\lambda}(\bar \rho-\rho)}
+1\right)
\left[1+
\max\left\{\log^+\left(
\dfrac{d_0}
{\underline{\lambda}\rho}\right),
\log^+\left( \frac{d_0} {\underline{\lambda}\bar \varepsilon}\right)
\right\}
\right]\right)
\end{align}
%
%\begin{align}
%\label{eq:c.alg3}
% \mathcal{O}\left(\left( 1 + \frac{\bar \lambda}{\underline \lambda} \right) \left(
%\dfrac{\left[1+1/\sqrt{1-\sigma^2}\right]d_0}
%{\bar{\lambda}(\bar \rho-\rho)}
%+\dfrac{1}{1-\sigma^2}\right)
%\left[1+
%\max\left\{\log^+\left(
%\left[\dfrac{1+\sigma}{1-\sigma}\right]\dfrac{d_0}
%{\underline{\lambda}\rho}\right),
%\log^+\left(\dfrac{\sigma^2 d_0}
%{(1-\sigma^2)\underline{\lambda}\bar \varepsilon}\right)
%\right\}
%\right]\right)
%\end{align}
iterations.
\end{theorem}
\begin{proof}
Note that at 
the $k$-th outer iteration of DR-HPE,  
we have $\mathcal{D}_0= 2^{k-1} \overline{\mathcal{D}}_0$. Moreover, in view of 
Theorem \ref{cr:c.alg2}(b),
DR-HPE terminates in at most $K$ outer iterations where $K$ is the smallest integer $k \ge 1$ satisfying 
$2^{k-1} \overline{\mathcal{D}}_0 \ge d_0$, i.e.,
\begin{align*}
%\label{eq:max.log}
 K=1+\left\lceil\log^+\left(
\dfrac{d_0}
{\overline{\mathcal{D}}_0}\right)\right\rceil.
%\lceil
\end{align*}
%
%
%\begin{align}
%\label{eq:ineq.k}
%K-1<\log\left(\dfrac{d_0}{\mathcal{D}_0}\right),
%\quad
% 2^{K}<\dfrac{2 d_0}
%{\mathcal{D}_0}. 
%\end{align}
%
Define
\begin{align}
\label{eq:bum}
\beta_1 &:=2+
\max\left\{\log^+\left(
\left[\dfrac{1+\sigma}{1-\sigma}\right]\dfrac{d_0^2}
{\underline{\lambda}^2\rho^2}\right),
\log^+\left(\dfrac{\sigma^2 d_0^2}
{2(1-\sigma^2)
\underline{\lambda}\bar \varepsilon}\right)
\right\}\\
%
%\beta_1 &=
%\dfrac{\left(1+1/\sqrt{1-\sigma^2}\right)
%\mathcal{D}_0}
%{2\underline{\lambda}(\bar\rho-\rho)}
%\left(2+\beta_2\right)\\
%
\label{eq:bz}
 \beta_0 &:=
\dfrac{\beta_1}{1-\sigma^2}
=
\dfrac{\left(1+1/\sqrt{1-\sigma^2}\right)
\overline{\mathcal{D}}_0}
{2\bar{\lambda}(\bar\rho-\rho)}
\beta_1
\end{align}
where the identity in \eqref{eq:bz} 
follows from \eqref{eq:def.dz}.
In view of Theorem \ref{cr:c.alg2}(a) and relations
\eqref{eq:bum}, \eqref{eq:bz}, we then conclude that 
the overall  number of inner iterations of 
DR-HPE is bounded by
\begin{align}
\label{eq:ktilde}
\widetilde K:=
\beta_0\sum_{k=1}^{K}\,\left(1+\frac{\bar \lambda}{\underline \lambda} 
2^{k-1}\right)=\beta_0\left[K + \frac{\bar \lambda}{\underline \lambda}( 2^{K} -1 )\right] \le \beta_0 \left( 1 + \frac{\bar \lambda}{\underline \lambda} \right)  2^{K}.
\end{align}
To prove the theorem, it suffices 
to show that $\widetilde K$ is bounded by \eqref{eq:c.alg3}.
Indeed, we consider two cases, namely, whether $K=1$ or $K>1$.
If $K=1$, then \eqref{eq:ktilde} 
implies that $\widetilde K \le 2\beta_0(1+\bar\lambda/\underline \lambda)$, and 
hence that
the order of $\widetilde K$ is bounded by \eqref{eq:c.alg3}
in view of the definition of $\beta_0$ in \eqref{eq:bz}.
Assume now that $K>1$ and note that the 
definition of $K$ implies that $k=K-1$ violates
the inequality $2^{k-1}\overline{\mathcal{D}}_0 
\ge d_0$, and hence 
that $2^{K}<4d_0/\overline{\mathcal{D}}_0$.
The latter conclusion and 
inequality \eqref{eq:ktilde} then imply that
$\widetilde {K}<4\beta_0d_0 (1+\bar\lambda/\underline \lambda)/
\overline{\mathcal{D}}_0$, 
which together with \eqref{eq:bum} and 
\eqref{eq:bz} then imply that $\widetilde K$ 
is bounded by \eqref{eq:c.alg3}.
\end{proof}

Note that  if the lower bound $\underline{\lambda}>0$  for the sequence of proximal stepsizes  is known,
and  $\bar \lambda=\underline{\lambda}$ and $\rho=\bar \rho/2$ are chosen as input for DR-HPE,
then the iteration-complexity bound \eqref{eq:c.alg3} reduces to bound \eqref{eq:it.c3}.
This observation justifies our claim 
preceding DR-HPE.

%\newpage
\section{Specific instances of the DR-HPE method}
\label{sec:app}

In this section, we briefly discuss 
specific ways of implementing step 1 of Algorithm 2.

More specifically, we assume that operator $B$ has the structure
\begin{equation}
\label{eq:inc.p.2}
B(x):=F(x)+C(x)
\end{equation}
where the following conditions hold:
\begin{itemize}
 \item[B.1)] $F:\mbox{Dom}(F)\subset X \to X$ 
 is a (single-valued) 
 monotone operator on 
 $\mbox{Dom}(C)\subset \mbox{Dom}(F)$, i.e.,
 \begin{align}
 \label{eq:mon.f}
  \inner{F(x)-F(x')}{x-x'}\geq 0,\quad 
 \forall x,x'\in \mbox{Dom}(C); 
 \end{align}
\item[B.2)] $F$ is $L$-Lipschitz continuous
 on a closed convex set $\Omega$ such that
 $\mbox{Dom}(C)\subset \Omega\subset \mbox{Dom}(F)$, 
 i.e., there exists $L>0$ such that
\begin{align}
\label{eq:lip}
 \|F(x)-F(x')\|\leq L\|x-x'\|\quad \forall x,x'\in 
\Omega;
\end{align}
\item[B.3)] $C:X\tos X$ is maximal monotone.
 \end{itemize}
Our goal in this section is to discuss a Tseng's modified forward-backward splitting (MFBS) type scheme for implementing
step 1 of Algorithm 2 for an operator $B$ with the above structure where two evaluations of $F$ and a single resolvent evaluation of $C$,
i.e., an operator of the form $(I+\lambda C)^{-1}$ 
for some $\lambda>0$, are made.

Let $(x_0,\sigma,\mu)$ be the first three entities of the input for Algorithm 2 and assume here that $\sigma \in (0,1)$.
Consider the MIP
%\eqref{eq:mip.mu}, i.e.,
%
\begin{align}
\label{eq:mip.mu2}
 0\in B_\mu(x):=F(x)+C_\mu(x)
\end{align}
where $ C_\mu:X\tos X$ is defined as
\begin{align}
\label{eq:def.fmu}
 C_\mu(x):=C(x)+\mu(x-x_0) \quad \forall x \in \Dom(C).
\end{align}
Given $x_{k-1} \in X$, the following two relations describes an iteration of a variant of Tseng's MFBS algorithm
studied in \cite{mon.sva-com.siam10} 
(see also \cite{mon.sva-hpe.siam10}) for the above MIP:

\begin{align}
 \label{eq:ts04} 
   y_k&=(I+\lambda C_\mu)^{-1}
 \big(x_{k-1}-\lambda F(P_{\Omega}(x_{k-1}))\big),\\
 \label{eq:ts05}  
 x_k&=y_k-\lambda  \big(F(y_k)-F(P_{\Omega}(x_{k-1}))\big)
\end{align}
where $\lambda:=\sigma/L$.
Since by assumption B.2 we have 
$\mbox{Dom}(C)\subset \Omega\subset \mbox{Dom}(F)$,
and $\mbox{Dom}(C_\mu)=\mbox{Dom}(C)$,
it follows that $P_{\Omega}(x_{k-1})$ and $y_k$
belong to $\mbox{Dom}(F)$, 
and hence that the iteration
defined in \eqref{eq:ts04}--\eqref{eq:ts05} 
is well-defined. Moreover, the assumption that
the resolvent of $C$ is computable makes the
resolvent $(I+\lambda C_\mu)^{-1}$ also computable
since 
\[
\left (I+\lambda C_\mu \right)^{-1}x=
\left (I+ \frac{\lambda}{1+\lambda\mu}C \right)^{-1} 
\left(\frac{x+\lambda\mu x_0}{1+\lambda\mu} \right) \quad x\in X.
\]

The following proposition was essentially proved in 
\cite[Proposition 4.5]{mon.sva-com.siam10} 
with a different notation.

\begin{proposition}
\label{pr:ts-hpe}
The points $y_k$ and $x_k$ defined by \eqref{eq:ts04} and \eqref{eq:ts05} with $\lambda=\sigma/L$ and
the vector
\begin{align*}
c_k &= \dfrac{1}{\lambda}(x_{k-1}-y_k) - F(P_{\Omega}(x_{k-1})) -\mu(y_k-x_0) 
\end{align*}
satisfy
 \begin{align}
\label{eq:ts-hpe01}
c_k\in C(y_k), \quad
 \norm{\lambda [F(y_k)+ c_k + \mu(y_k-x_0)]+y_k-x_{k-1}}
 \leq \sigma\norm{y_k-x_{k-1}},
\end{align}
and hence $b_k:=F(y_k)+c_k$, $\lambda_k:=\lambda$ and $\varepsilon_k:=0$ satisfy \eqref{eq:ec.2}.
\end{proposition}
\begin{proof}
The inclusion in \eqref{eq:ts-hpe01} follows directly
from \eqref{eq:ts04}, \eqref{eq:def.fmu}
and the definition of $c_k$.
On the other hand, using items (a) and (c)
of \cite[Proposition 4.5]{mon.sva-com.siam10}  
(with a different notation), we obtain 
the inequality in \eqref{eq:ts-hpe01}.
The last statement of the proposition follows
from the definition of $b_k$, \eqref{eq:inc.p.2},
\eqref{eq:ts-hpe01} and 
Proposition \ref{pr:teps.pr}(d). 
\end{proof}

In the next theorem we show the iteration-complexity
of DR-HPE for solving \eqref{eq:inc.p.2} 
under the assumption that the iteration 
of the variant of Tseng's MFBS method
described in \eqref{eq:ts04}--\eqref{eq:ts05} 
is used
as an implementation of step 1 of Algorithm 2.

\begin{theorem}
\label{th:main.ts}
If
$\max\{ \sigma^{-1},(1-\sigma)^{-1}\}=\mathcal{O}(1)$, then
\emph{DR-HPE} in which step 1 of \emph{Algorithm 2} is implemented according to the
recipe described in Proposition \ref{pr:ts-hpe} terminates
with a pair 
$(y,b)$ satisfying
\begin{align}
\label{eq:ts.pc}
 b\in (F+C)(y),\quad \norm{b}\leq \bar\rho
\end{align}
in at most
\begin{align}
\label{eq:c.alg4}
 \mathcal{O}\left(\left(1+
\dfrac{Ld_0}
{\bar \rho-\rho}
\right)
\left[1+
\log^+\left(
\dfrac{Ld_0}
{\rho}\right)\right]\right)
\end{align}
iterations where $\bar \rho$ and $\rho$ are as in 
step 0 of \emph{DR-HPE}.
\end{theorem}
\begin{proof}
 The result is a direct consequence of Theorem 
\ref{th:main} and Proposition \ref{pr:ts-hpe}
where $\underline{\lambda}=
\bar \lambda=\lambda:=\sigma/L$.
We note that since by Proposition \ref{pr:ts-hpe}
we have $\varepsilon_k=0$ for all $k\geq 1$
the complexity bound on \eqref{eq:c.alg3} is 
independent of the precision $\bar\varepsilon>0$.
\end{proof}

We now make some comments about the special 
instance of DR-HPE described in Theorem \ref{th:main.ts} in light of
a previous variant of Tseng's MFBS algorithm studied in \cite{mon.sva-com.siam10} for solving MIP \eqref{eq:inc.p.2}.
First, the cost of an inner iteration of the above two methods are identical.
Second, if $\rho=\bar \rho/2$, then the complexity bound \eqref{eq:c.alg4}
reduces to
\begin{equation}
\label{eq:bd.ts}
 \mathcal{O}\left(\left(1+
\dfrac{Ld_0}
{\bar \rho}
\right)
\left[1+
\log^+\left(
\dfrac{Ld_0}
{\bar \rho}\right)\right]\right),
\end{equation}
which
improves the
pointwise iteration-complexity bound
$\mathcal{O}\left((Ld_0/\bar\rho)^2\right)$
for the variant of Tseng's MFBS
algorithm (see \cite[Theorem 4.6]{mon.sva-com.siam10}).
Third, it is proved in 
\cite[Theorem 6.2(b)]{mon.sva-hpe.siam10}
that the Tseng's MFBS variant finds an ergodic pair $(b,y)$ satisfying
$b\in (F+C)^\varepsilon(y)$, $\norm{y}\leq \bar\rho$
and $\varepsilon\leq \bar\varepsilon$
in at most    
$\mathcal{O}\left(\max\left[Ld_0/\bar\rho,
Ld_0^2/\bar\varepsilon\right]\right)$ iterations.
Note that the dependence of the latter bound on $\bar \rho$
differs from the one in \eqref{eq:bd.ts} only by a logarithmic term. Moreover, in contrast to the latter bound,
\eqref{eq:bd.ts} does not depend on $\bar \varepsilon$. Also, the error criterion 
implied by the latter ergodic result 
is weaker than the one in \eqref{eq:ts.pc}.
In summary, Theorem 4.2 establishes a pointwise iteration-complexity bound which closely approaches the latter ergodic bound
while guaranteeing at the same time an error criterion stronger than the one for the aforementioned ergodic result.

We finish this section by noting that,
if $C=\partial g$ where $g: X \to (-\infty,\infty]$ is a proper closed convex
function,  then an iteration of  Korpelevich's extragradient algorithm (see for example Section 4 of \cite{mon.sva-com.siam10})
can also be used to implement step 1 of  Algorithm 2 and, as a consequence, yields a different instance of DR-HPE.
Clearly, it is possible to derive a result for the new variant similar to Theorem \ref{th:main.ts} in which
the error criterion becomes 
$b\in (F+\partial_\varepsilon g)(y)$,
$\norm{b}\leq \bar\rho$, $\varepsilon\leq 
\bar\varepsilon$ 
and the complexity bound is given by 
\eqref{eq:c.alg3}
(and hence depends on $\varepsilon$) 
with
$\bar\lambda=\underline{\lambda}=\lambda:=\sigma/L$. 
Note that the latter error criterion, while weaker than the one in \eqref{eq:ts.pc}, 
is still stronger than the one
of the ergodic result for the Tseng's MBFS variant
(see, for instance, \cite[Corollary 5.3(b)]{mon.sva-hpe.siam10}).

%\section{Concluding remarks}
%\label{sec:cr}

\appendix

%\newpage
\section{Proof of Proposition \ref{pr:mmm}}
\label{sec:app01}

From now on  
$\{x_k\}$, $\{y_k\}$, $\{v_k\}$, 
$\{\lambda_k\}, \{\sigma_k\}$ and 
$\{\varepsilon_k\}$
are sequences generated by Algorithm 1.

Define, for $k\geq 1$:
\begin{align}
  \label{eq:gammak}
  \gamma_k:X\to\R,\;\;
  \gamma_k(x):=\inner{v_k}{x-y_k}-
\varepsilon_k\quad \forall x\in X
\end{align}
and 
\begin{align}
  \label{eq:muk}
   %\mu_0=1,\quad \mu_{k}=\prod_{i=1}^k\,(1+2\lambda_i\eta)\,.%\;\;k=1,2,\dots
	\theta_k:=\left(\dfrac{1}{2\lambda_k\mu}+\dfrac{1}{1-\sigma^2}\right)^{-1}\in (0,1).
\end{align}

\begin{proposition}
  \label{pr:b1}
Let $\gamma_k(\cdot)$ and 
$\theta_k$ be as in~\eqref{eq:gammak}
and~\eqref{eq:muk}, respectively.
For every $k\geq 1$:
\begin{itemize}
  \item[\emph{(a)}] 
  \label{it:b1} 
  $x_k=\arg\min\lambda_k
 \gamma_k(x)+\norm{x-x_{k-1}}^2/2$;
 \item[\emph{(b)}] 
\label{it:b2}
  $\min\lambda_k\gamma_k(x)+\norm{x-x_{k-1}}^2/2
     \geq(1-\sigma^2)\norm{y_k-x_{k-1}}^2/2$;
 \item[\emph{(c)}]
 \label{it:a3} 
 $\gamma_k(x^*)\leq -\mu\norm{x^*-y_k}^2$ for any
     $x^*\in(A+B)^{-1}(0)$;
 \item[\emph{(d)}] \label{it:b3} for any $x^*\in (A+B)^{-1}(0)$,
     \begin{align*}
       \norm{x^*-x_{k-1}}^2\geq 2\lambda_k\mu\norm{x^*-y_k}^2
       +(1-\sigma^2)\norm{y_k-x_{k-1}}^2+\norm{x^*-x_k}^2
     \end{align*}
     and
    \begin{align*}
       \left(1
       -\theta_k\right)
       \norm{x^*-x_{k-1}}^2\geq \norm{x^*-x_k}^2.
     \end{align*} 
		%\begin{align*}
    %   \left(1
    %   -\dfrac{2\eta\lambda_k(1-\sigma^2)}{2\eta\lambda_k+1-\sigma^2}\right)
    %   \norm{x^*-x_{k-1}}^2\geq \norm{x^*-x_k}^2.
    % \end{align*}
   \end{itemize}
\end{proposition}

\begin{proof}
  (a) This statement follows trivially from~\eqref{eq:es}
  and~\eqref{eq:gammak}.
  
  (b) Direct use of 
  \eqref{eq:es} and \eqref{eq:gammak} yields, 
  after  trivial algebraic manipulations,
 \begin{align*}
    \lambda_k\gamma_k(x_k)+
    \dfrac12\norm{x_k-x_{k-1}}^2=
    \dfrac12\left[\norm{y_k-x_{k-1}}^2
    -\left(\norm{\lambda_kv_k+y_k-x_{k-1}}^2+
   2\lambda_k\varepsilon_k\right)\right],
  \end{align*}
  which, combined with 
  item (a) and \eqref{eq:ec} proves item (b).
  
  (c) If $x^*\in (A+B)^{-1}(0)$, 
  then there exists $a^*\in A(x^*)$ such
  that $-a^*\in B(x^*)$.  
  It follows from the inclusion
  in~\eqref{eq:ec} that there exists 
  $a_k\in A(y_k)$,
  $b_k\in B^{[\varepsilon_k]}(y_k)$ 
  such that $v_k=a_k+b_k$. It follows
  from these inclusions, assumption A.1,
  and~\eqref{eq:def.eps}  that
  \[
  \inner{a^*-a_k}{x^*-y_k}\geq\mu\norm{x^*-y_k}^2,\qquad
  \inner{b_k+a^*}{y_k-x^*}\geq-\varepsilon_k\,.
  \]
  To end the proof of item (c), add these inequalities, observe that
  $a_k+b_k=v_k$, and use the definition 
\eqref{eq:gammak}.
  
  (d) 
  It follows from~\eqref{eq:gammak}, 
  (a), and (b) that, for all
  $x\in X$
  \begin{eqnarray*}
  \lambda_k\gamma_k(x)+\frac{1}{2}\|x-x_{k-1}\|^2
  &=&
  \left(\min\lambda_k\gamma_k(x)+
  \frac{1}{2}\|x-x_{k-1}\|^2\right)+
  \frac{1}{2}\|x-x_k\|^2\,\\
  &\geq& 
  \frac{1}{2}\left((1-\sigma^2)\|y_k-x_{k-1}\|^2+
  \|x-x_k\|^2\right).
  \end{eqnarray*}
  To prove the first inequality in item (d) take
  $x=x^*\in (A+B)^{-1}(0)$ in the above equation and use item (c).
  To prove the second inequality, observe that $\norm{x^*-y_k}+\norm{y_k-x_{k-1}}
  \geq \norm{x^*-x_{k-1}}$,
  \begin{align*}
    \min\{(1-\sigma^2)r^2+2\mu\lambda_ks^2\;|\; r,s\geq 0,\;
    r+s\geq\norm{x^*-x_{k-1}}\}=
    \theta_k\norm{x^*-x_{k-1}}^2
		%\dfrac{2\eta\lambda_k(1-\sigma^2)}{2\eta\lambda_k+1-\sigma^2}
  \end{align*}
  and use the first inequality of item (d).
\end{proof}
The following Lemma follows trivially from 
the inequality in~\eqref{eq:ec}, the use of the 
triangle inequality and 
the fact that $\varepsilon_k\geq 0$.
\begin{lemma}\label{lm:1}
For $k\geq 1$:
 \[
  \left(1-\sigma_k\right)\|y_k-x_{k-1}\|\leq\|\lambda_kv_k\|\leq \left(1+\sigma_k\right)
	\|y_k-x_{k-1}\|. 
 \] 
\end{lemma}
%\begin{proof}
%  Use the inequality in~\eqref{eq:ec}, triangle inequality and 
%	the fact that $\varepsilon_k\geq 0$. 
%\end{proof}
%
%

In the next proposition, 
we establish rates of convergence for the sequences 
$\{x_k\}$, $\{v_k\}$ and $\{\varepsilon_k\}$ 
generated by Algorithm 1.
	%
	
%\newpage
\begin{proposition}
  \label{th:pt}
Let $d_0$ denote the distance
$x_0$ to the solution set of~\eqref{eq:inc.p}
and define for every $k\geq 1$:
\begin{align}
\label{eq:def.Gamma}
 \Gamma_k:=\left[\prod_{j=1}^k\,(1-\theta_j)\right]^{1/2}.
\end{align}
 Then, for every 
 $k\geq 1$, $v_k\in A(y_k)+B^{[\varepsilon_k]}(y_k)$
 and  
\begin{align}
\label{eq:th.1}
  \norm{v_k}\leq \sqrt{\dfrac{1+\sigma}{1-\sigma}} 
   \left(\dfrac{\Gamma_{k-1}}{\lambda_{k}}\right)d_0,\quad 
\varepsilon_k\leq \frac{\sigma^2}{2(1-\sigma^2)}\left(\dfrac{\Gamma_{k-1}^2}{\lambda_k}\right)d_0^2,
\end{align}
\begin{align}
\label{eq:th.5}
\|x^*-x_k\|\leq \Gamma_k\|x^*-x_0\|
  \qquad \forall x^*\in (A+B)^{-1}(0).
\end{align}
\end{proposition}
\begin{proof}
First note that~\eqref{eq:th.5} follows 
from
the second inequality in Proposition~\ref{pr:b1}(d) 
and~\eqref{eq:def.Gamma}.
Using the first inequality
in Proposition~\ref{pr:b1}(d)
and~\eqref{eq:th.5}, 
we conclude that, for all $x^*\in (A+B)^{-1}(0)$,
\[
(1-\sigma^2) \|y_k-x_{k-1}\|^2 \le 
\Gamma_{k-1}^2 \norm{x^*-x_0}^2\qquad \forall k\geq 1.
\]
Note now that \eqref{eq:th.1} follows
from the latter inequality and the relations
\[
\varepsilon_k \le \frac{\sigma^2 \|y_k-x_{k-1}\|^2}{2\lambda_k},
 \ \ \ \
\|v_k\| \le \frac{(1+\sigma)\|y_k-x_{k-1}\|}{\lambda_k},
\]
which are due to  \eqref{eq:ec} and the second inequality in Lemma~\ref{lm:1}.
\end{proof}

\noindent
{\bf Proof of Proposition \ref{pr:mmm}.}
The assumption $\lambda_k\geq \underline{\lambda}>0$
for every $k\geq 1$ and the fact that
 the scalar function
\[
t >0 \mapsto \left( \frac{1}{2t\mu} +
 \frac{1}{1-\sigma^2} \right)^{-1}
\]
is nondecreasing, combined with \eqref{eq:muk} 
and \eqref{eq:def.theta}, imply that $\theta_k\geq \theta$
for all $k\geq 1$. 
From  the latter inequality and \eqref{eq:def.Gamma}
we obtain $\Gamma_k \le (1-\theta)^{k/2}$ 
for every $k\geq 1$, which, in turn,
combined with Proposition~\ref{th:pt} completes 
the proof.

%\bibliographystyle{plain}
%\bibliography{prox}

\def\cprime{$'$}

%\bibliographystyle{plain}
% \bibliography{/Users/testuser/Dropbox/papers/bibfiles/svaiter}

\end{document}